\documentclass[a4paper,11pt,reqno,draft]{amsart}
\usepackage{amsmath, amssymb, amsthm, amsfonts, amsgen, stmaryrd, mathrsfs}
\usepackage[centertags]{amsmath}
\usepackage{indentfirst}
\usepackage{fancyhdr}
\usepackage{color}
\numberwithin{equation}{section}
\newcommand{\dx}{\,dx}
\newcommand{\ds}{\displaystyle}
\let \l =\lambda
\newcommand{\R}{\mathbb{R}}

\newcommand{\om}{\Omega}
\newcommand{\Div}{{\rm div}\,}
\newcommand{\curl}{{\rm curl}\,}
\newcommand{\cof}{{\rm cof}\,}
\newcommand{\adj}{{\rm adj}\,}
\newcommand{\ul}{u_{\lambda}}
\newcommand{\1}{{\bf 1}}
\newcommand{\tr}{{\rm tr}\,}
\newcommand{\eps}{\epsilon}
\newcommand{\sca}{\mathcal{A}}
\newcommand{\dist}{\textrm{dist}\,}
\newcommand{\scl}{\mathcal{L}}
\newcommand{\Det}{\mathrm{Det}\,}

\newcommand{\im}{\mathrm{im}\,}
\newcommand{\lam}{\lambda}
\newcommand{\scg}{\mathcal{G}}
\newcommand{\rb}{\bar{\rho}}
\newcommand{\rast}{\rho^{\ast}}
\newcommand{\atr}{\rm{atr}}
\newcommand{\scf}{\mathcal{F}}

\newtheorem{definition}{Definition}[section]
\newtheorem{lemma}[definition]{Lemma}
\newtheorem{theorem}[definition]{Theorem}
\newtheorem{proposition}[definition]{Proposition}

\newtheorem{remark}[definition]{Remark}

\begin{document}
\title[A simple sufficient condition for quasiconvexity]
{A simple sufficient condition for the quasiconvexity of elastic stored-energy functions in spaces which allow for cavitation}
\author[J.J. Bevan]{Jonathan J. Bevan}
\author[C.I. Zeppieri]{Caterina Ida Zeppieri}
\address{Department of Mathematics, University of Surrey, Guildford, GU2 7XH, United Kingdom}
\email{j.bevan@surrey.ac.uk}
\address{Institut f\"ur Angewandte Mathematik, Universit\"at M\"unster, Einsteinstr. 62, 48149 M\"unster, Germany}
\email{caterina.zeppieri@uni-muenster.de}

\begin{abstract}
In this note we formulate a sufficient condition for the quasiconvexity at $x \mapsto \l x$ of certain functionals $I(u)$ which model the stored-energy of elastic materials subject to a deformation $u$.  The materials we consider may cavitate, and so we impose the well-known technical condition (INV), due to M\"{u}ller and Spector, on admissible deformations.  Deformations obey the condition 
$u(x)= \lambda x$ whenever $x$ belongs to the boundary of the domain initially occupied by the material.  In terms of the parameters of the models, our analysis provides an explicit upper bound on those $\lambda>0$ such that $I(u) \geq I(u_{\lambda})$ for all admissible $u$, where $u_{\lambda}$ is the linear map $x \mapsto \lambda x$ applied across the entire domain.  This is the quasiconvexity condition referred to above. 
\end{abstract}

\maketitle

\begin{center}
\begin{minipage}{12cm}
\small{ 

}
\end{minipage}
\end{center}

\section{Introduction}

\noindent Since the seminal work of Ball \cite{Ba82}, the phenomenon of cavitation in nonlinear elasticity has been studied by many authors, with significant advances \cite{MH2010,MH2011,MS95} having been
made in the case that an appropriately defined surface energy be part of the cost of deforming a material.   In this note we consider the original case of a purely bulk energy
\begin{equation}\label{j:w}I(u)=\int_{\om} W(\nabla u(x))\,dx,\end{equation}
where as usual $u: \om \subset \R^{n} \to \R^{n}$ represents a deformation of an elastic material occupying the domain $\om$ in a reference configuration, and where $n=2$ or $n=3$.  Our goal is to give a straightforward, explicit characterization of those affine boundary conditions of the form 
\[ \ul(x):=\l x,\]
where $\l$ is a positive parameter, which obey the quasiconvexity inequality\footnote{Strictly speaking, this is a $W^{1,q}$-quasiconvexity inequality; the term quasiconvexity usually refers to the case in which $I(u) \geq I(\ul)$ holds for all \emph{Lipschitz} $u$ agreeing with $\ul$ on $\partial \om$. See, e.g., \cite{BM84} for the distinction.} 
\begin{equation}\label{j:qc} I(u) \geq I(\ul).\end{equation}
In the case of radial mappings \cite{Ba82} it is this inequality which must be violated in order that a global minimizer of $I$ might cavitate (i.e.\ where a hole is created in the deformed material), a crucial ingredient of which is the application of a large enough stretch on $\partial \om$ (i.e.\ taking $\l$ sufficiently large).    When deformations are not restricted to any particular type we are still interested in whether the quasiconvexity inequality holds for a given $\l$ since it rules out the possibility that a global energy minimizer cavitates.  Thus the largest $\l$ for which \eqref{j:qc} holds is sometimes referred to as a critical load.   Our chief inspiration for this work is \cite{MSS96}, where bounds for the critical load are given in terms of constants appearing in certain isoperimetric inequalities.   We use a different technique to find an explicit upper bound on the critical load in the two and three dimensional settings.

Our method also yields conditions on $\nabla u$ for the inequality \eqref{j:qc} to be close to an equality in the sense that if $\delta(u):=I(u)-I(u_{\lambda})$ is small and positive then, in the two dimensional case
\begin{equation}\label{j:e2dintro} \int_{\om}\min\{|\nabla u- \l\1|^2,|\nabla u - \l \1|^q\}\,dx \leq c\, \delta(u),\end{equation}
where $1<q<2$ is an exponent governing the growth of the stored-energy function $W$ appearing in \eqref{j:w}.  See Theorem \ref{suffthm2} for the latter.    The corresponding condition in three dimensions is 
\[\int_{\om} |\nabla u - \l \1|^{q} \,dx \leq c \delta(u),\]
where $2<q <3$:  see Theorem \ref{c:rig-3d} for details.  In both cases the Friesecke, James and M\"uller rigidity estimate \cite[Theorem 3.1]{FJM02} (see also \cite[Theorem 1.1]{CDM12}) is used in conjunction with the boundary condition to recover information apparently lost in deriving sufficient conditions for \eqref{j:qc}.    We also note that these conditions are invariant under the elasticity scaling in which a function $v(x)$, say, is replaced\footnote{this is an oversimplification: see 
\cite[Proposition 2.3]{BM84} or \cite{SS02} for full details} by $v^{\eps}(x)=\frac{1}{\eps}v(\eps x)$, where $\eps > 0$.
This is important in view of the example in \cite[Section 1]{SS02}.  The latter says, among other things, that, in the absence of surface energy, a deformation which cavitates at just one point in the material can have the same energy as another deformation with infinitely many cavities.  

The setting we work in is motivated by \cite{MS95} in the sense that we impose condition (INV), a topological condition which is explained later.  Cavitation problems must be posed in function spaces containing discontinuous functions. In particular, Sobolev spaces of the form $W^{1,q}(\om,\R^{n})$ with $q \geq n$ are not appropriate, since their members are necessarily continuous.   In the case $q>n$ this follows from the Sobolev embedding theorem, while if $q=n$ then well-known results \cite{VG77,Sv88}, applying to maps $u$ with $\det \nabla u> 0$ a.e., imply that $u$ has a continuous representative.  Thus we work in $W^{1,q}(\om,\R^{n})$, where $n-1 < q < n$, and in so doing we are able to take advantage of existing results, including but not only those of \cite{MS95}.

The stored-energy functions we consider in the two dimensional case have the form 
\[ W(A) := |A|^{q} + h(\det A)\]
where $1 < q < 2$ and where $h \colon \R \to [0,+\infty]$ satisfies 
\begin{itemize}
\item[(H1)] $h$ is convex and $C^{1}$ on $(0,+\infty)$;
\item[(H2)] $\lim_{t \to 0+} h(t) = +\infty$ and $\liminf_{t \to \infty} \frac{h(t)}{t} > 0$; 
\item[(H3)] $h(t) = + \infty$ if $t \leq 0$.
\end{itemize}

In three dimensions the appropriate class of $W$ is detailed in Section \ref{s3}.   In both cases we define a set  of admissible deformations
\begin{equation}\label{j:defalam}\sca_{\l}:=\{u \in W^{1,q}(\om,\R^{n}): \ u = \ul \ \textrm{on} \ \partial \om, \ \det \nabla u > 0 \ \rm{a.e.\, in}\; \om\}.\end{equation}
It is made clear in \cite{Ba82} and \cite{Si86} that when $\lambda$ is sufficiently large there are maps $u_{0}$ belonging to $\sca_{\lambda}$ of the form
\[ u_{0}(x) = r(|x|)\frac{x}{|x|},\]
with $r(0) > 0$, such that 
\begin{equation}\label{notqc1}I(u_{0}) < I(u_{\lambda}).\end{equation}
The growth of $h(t)$ for large values of $t$ is pivotal in ensuring that such an inequality can hold.   Thus the integrand $W$ is not ($W^{1,q}$-)quasiconvex at $\lambda \1$.   The loss of quasiconvexity is typically associated with so-called cavitating maps like $u_{0}$, whose distributional Jacobian $\Det \nabla u_{0}$ is proportional to a Dirac mass, a remark first made by Ball in \cite{Ba82}.

For later use, we recall that the distributional Jacobian of a mapping in $W^{1,p}(\om,\R^n)$, with $p > n^2/(n+1)$, is defined by 
\[ (\Det \nabla u)(\varphi) = -\frac{1}{n} \int_{\om} \nabla \varphi \cdot (\adj \nabla u)u \,dx,\]  
where $\varphi$ belongs to $C_{0}^{\infty}(\om)$.   When $u$ is $C^{2}$ the distributional Jacobian coincides with the Jacobian $\det \nabla u$.  The same is true if, more generally, $u \in W^{1,p}(\om)$ with $p \geq n^{2}/(n+1)$ and $\Det \nabla u$ is a function (see \cite{Mu90}).

The paper is arranged as follows: after a short explanation of notation, we consider the two and three dimensional cases separately in Sections \ref{s2} and \ref{s3} respectively.  Subsection \ref{s2:error} contains the bulk of the estimates needed for \eqref{j:e2dintro}; the relevant estimates in the three dimensional case draw on these results and are presented succinctly in Section 3.      Along the way, we give a slight improvement of \cite[Lemma 2.15]{Zh91}, and, as a byproduct of our work in three dimensions we are led to a conjecture concerning the quasiconvexity of a certain function which, to the best of our knowledge, has not yet been considered in the literature.

\subsection{Notation}

We denote the $n \times n$ real matrices by $\mathbb{R}^{n \times n}$ and the identity matrix by $\1$.  Throughout, $\om\subset \R^n$ is a fixed, bounded domain with Lipschitz boundary, $B(a,R)$ represents the open ball in $\mathbb{R}^{n}$ centred at $a$ with radius $R>0$ and $S(a,R):=\partial B(a,R)$.   
Other standard notation includes $\scl^n$ for the Lebesgue measure in $\R^n$.

The inner product of two matrices $A,B \in \mathbb{R}^{n \times n}$ is $A \cdot B := \tr(A^{T}B)$.   
This obviously holds for vectors too.  Accordingly, we make no distinction between the norm of a matrix and that of a vector:  both are defined by $|\nu|:=(\nu \cdot \nu)^{\frac12}$.     For any $n \times n$ matrix we write $\adj A:= (\cof A)^{T}$, while $\tr A$ and $\det A$ denote, as usual, the trace and determinant of $A$, respectively.   Other notation will be introduced when it is needed.


\section{The two dimensional case}\label{s2}
\noindent The relevance of the distributional Jacobian to the loss of quasiconvexity can be seen using the following argument, the first part of which is due originally to Ball \cite{Ba77}. Firstly, the convexity of $A \mapsto |A|^{q}$ and of $h$ implies that 
\[ W(\nabla u) \geq W(\lambda \1) + q |\lambda \1|^{q-2}\lambda 1 \cdot (\nabla u - \lambda \1)+h'(\lambda^{2})(\det \nabla u- \lambda^{2}),\]
which, when $u \in \sca_{\lambda}$, can be integrated over $\om$;  the result is
\begin{equation}\label{halflife}I(u) \geq I(u_\lambda) +  h'(\l^{2}) \int_{\om} (\det \nabla u - \det \nabla \ul)\, dx.\end{equation}
Clearly, if the integral with prefactor $h'(\lambda^{2})$ vanishes, 
that is if
\begin{equation}\label{lassus1} \int_{\om} (\det \nabla u - \det \nabla \ul) \, dx = 0,\end{equation}
then $I(u) \geq I(\ul)$ follows.  This can be ensured, for example, by imposing further conditions on $u$ guaranteeing that 
\begin{equation}\label{changevar} \int_{\om} f(u(x)) \, \det \nabla u(x)\,dx = \int_{\R^2}f(y) \, \textrm{deg}(\bar{u},\partial \om,y)\,dy\end{equation}
for any bounded continuous function $f$, where $\bar{u}$ represents the trace of $u$, here assumed to possess a continuous representative in order that the degree is well-defined.  The idea behind this originates in \v{S}ver\'{a}k's work \cite{Sv88}, and was later refined by M\"uller, Qi and Yan \cite{MQT94}.  \v{S}ver\'{a}k showed, among other things, that functions obeying \eqref{changevar} have a continuous representative, and in particular cannot cavitate.\footnote{One could also produce \eqref{changevar} without reference to either of these papers.  For example, \eqref{lassus1} holds whenever $u$ is continuous, satisfies Lusin's $N$-property (i.e., $u$ maps sets of (Lebesgue) measure zero to sets of (Lebesgue) measure zero), and $\det \nabla u$ belongs to $L^{1}(\om)$.  See, for example, \cite[Theorem 5.25]{FG95}.}
In fact, the discrepancy between $\int_{\om} \det \nabla u \dx$ and $\int_{\om} \det \nabla u_{\lambda}\dx$ can be measured using $\Det \nabla u$ and interpreted in terms of cavitation provided some additional conditions are imposed on $u$.  Initially following the approach in \cite{MSS96}, we appeal to a result in \cite{MS95}.  Before that, we recall the definition of M\"{u}ller and Spector's condition (INV), stated in terms of a general dimension $n$ and domain $\om$. 

\begin{definition}(\cite[Definition 3.2]{MS95}) The map $u: \om \to \R^n$ satisfies condition (INV) provided that for every $a \in \om$ there exists an $\scl^1$-null set $N_{a}$ such that, for all $R \in (0,\dist(a,\partial \om))\setminus N_{a}$, $u|_{S(a,R)}$ is continuous, 
\begin{itemize}\item[(i)] $u(x) \in \im_{T}(u,B(a,R))\cup u(S(a,R))$ for $\scl^{n}$-a.e. $x \in \overline{B(a,R)}$, and
\item[(ii)] $u(x) \in \R^n \setminus \im_{T}(u,B(a,R))$ for $\scl^{n}$-a.e. $x \in \om \setminus B(a,R)$. 
\end{itemize}
\end{definition}

The topological image of $B(a,R)$ under the mapping $u$, $\im_{T}(u,B(a,R))$, is defined below.

\begin{lemma}(\cite[Lemma 8.1]{MS95}.)
Let $u \in W^{1,q}(\om;\R^n)$ with $q > n-1$.  Suppose that $\det \nabla u > 0$ a.e. in $\om$ and that $u^\ast$, the precise representative\footnote{See \cite[p. 13]{MS95} for a definition of $u^\ast$.} of $u$, satisfies condition (INV). Then $\Det \nabla u \geq 0$ and hence $\Det \nabla u$ is a Radon measure.  Furthermore,
\begin{equation}\label{MS8.11}
\Det \nabla u = \det \nabla u \,\mathcal{L}^{n} + m 
\end{equation}  
where $m$ is  singular with respect to Lebesgue measure and for $\scl^{1}\hbox{-}a.e. \ R \in (0,\dist(a,\partial \om))$,
\begin{equation}\label{topvol}(\Det \nabla u )(B(a,R)) = \scl^{n}(\im_{T}(u,B(a,R))).\end{equation}  
\end{lemma}
\begin{remark}\rm{Under the assumption that the perimeter of $\im_{T}(u,\om)$ is finite it can be shown that the singular part of $\Det \nabla u$ is a sum of Dirac masses.  Thus the left-hand side of \eqref{lassus2} below is 
$-1 \times$ (volume of cavities created by the deformation $u$). See \cite[Theorem 8.4]{MS95} for more details.}\end{remark}
\begin{remark}\label{mpos}\rm{Since $m$ is singular with respect to Lebesgue measure, and in view of $\Det \nabla u \geq 0$, it is clear that $m \geq 0$.}\end{remark}

Reverting to the two dimensional case $\om \subset \R^2$, the assumption that $u \in W^{1,q}(\om)$ for $q>1$ implies (by Sobolev embedding) that $u\arrowvert_{S(a,R)}$ is continuous for $\scl^{1}$-a.e. $R \in (0,\dist(a,\partial \om))$.  Hence, for such $R$, the topological image 
\[\im_{T}(u,B(a,R)) = \{y \in \R^2 \setminus u(S(a,R)): \ \deg(u,S(a,R),y) \neq 0\}\]
is well-defined.    Following \cite{MSS96}, we extend $u$ by setting it equal to $\ul$ on $B(0,M)\setminus \bar{\om}$, where $M$ is chosen so that $\bar{\om} \subset B(0,M)$, and we assume that the extension satisfies condition (INV) on $B(0,M)$.    It is then straightforward to check, using the definition of the distributional Jacobian, its representation through \cite[Lemma 8.1]{MS95} and \eqref{topvol}, that 
\begin{equation}\label{lassus2} -m(\bar{\om}) = \int_{\om} (\det \nabla u - \det \nabla \ul) \, dx \end{equation}
Finally, by applying \eqref{lassus2} to inequality \eqref{halflife}, we obtain
\begin{equation}\label{halflife1}I(u) \geq I( u_\lambda ) - h'(\lambda^2)\, m(\bar{\om}).\end{equation}
It is clear that when $h'(\l^2) \leq 0$ or $m(\bar{\om})=0$ we have $I(u) \geq I(\ul)$.  Summarising the above, we have the following: 

\begin{proposition}\label{p1} Suppose that $W(A)=|A|^{q} + h(\det A)$, where $h$ satisfies $(H1)-(H3)$, and where $q > 1$.  Let $B(0,M)$ contain $\bar{\om}$ and denote by $u^{\textrm{e}}$ the extension of $u$ to $B(0,M) \setminus \om$ defined by
\[u^{\textrm{e}}(x) := \left\{\begin{array}{l l} u(x) & \textrm{if} \ x \in \om, \\ 
\ul(x)  & \textrm{if} \ x \in B(0,M) \setminus \om. \end{array}\right.\]
Assume that $u^{\textrm{e}}$ satisfies the hypotheses of \cite[Lemma 8.1]{MS95} in the case that $n=2$. Then if $\int_{\om} \det \nabla u \,dx= \int_{\om}\det \nabla \ul \,dx$ or if $h'(\l^{2}) \leq 0$, the inequality $I(u) \geq I(\ul)$ holds.   
\end{proposition}
%

The rest of this section handles the case $h' (\lambda^2) > 0$ 
and $m(\bar{\om}) > 0$,  where $m$ is given by \eqref{lassus2},  which is the situation not covered by Proposition \ref{p1}. 
The following is a slightly improved version of a lemma by Zhang which, although stated here for general $n$, will only be needed in the case $n=2$.
\begin{lemma}\label{l2}(Adaptation of \cite[Lemma 2.15]{Zh91})  For $1 < q < 2$, $M > 0$,  and $A, B \in \R^{n \times n}$ with $0 < |A| \leq M$,
\begin{displaymath}
|A+B|^q - |A|^q - q |A|^{q-2} A \cdot B  \geq \left\{\begin{array}{l l} C_{1}(M,q) |B|^2 & \textrm{if} \ |B| \leq M, \\
C_{2}(q) |B|^q & \textrm{if} \ |B| \geq M, \end{array}\right.
\end{displaymath}
The constants $C_{1}(M,q)$ and $C_{2}(q)$ are given by 
\begin{align}\label{c1} C_{1}(M,q) & =  \frac{1}{2 (2M)^{2-q}}, \\
\label{c2} C_{2}(q) & =  \frac{1}{2(2^{2-q})}.
 \end{align}
\end{lemma}
\begin{proof} The only part which requires proof is the constant $C_{2}(q)$ since it is larger than the original version $\tilde{C}_{2}(q):=\frac{1}{2(3^{2-q})}$ given in \cite[Lemma 2.14]{Zh91}.   The constant $\tilde{C}_{2}(q)$ appears in \cite[Eq. (2.23)]{Zh91} as a prefactor in the estimate
\[ \int_{0}^{1} \frac{(1-s)|B|^2}{|A+sB|^{2-q}}\,ds \geq \tilde{C}_{2}(q)|B|^q\]
under the assumption that $|B| \geq M$.  Now, in terms of $\tau: = |B|/M$,
\begin{eqnarray*} \frac{(1-s)|B|^2}{|A+sB|^{2-q}} & \geq & \frac{(1-s)|B|^2}{|M+s|B||^{2-q}} \\
& = &  \frac{(1-s) M^q \tau^2}{(1+s\tau)^{2-q}} \\
& \geq & \frac{(1-s) \tau^{2-q} |B|^q}{(1+\tau)^{2-q}.}
\end{eqnarray*}
Since $\tau \geq 1$, the quantity $\frac{\tau^{2-q}}{(1+\tau)^{2-q}}$ is bounded below by $1/2^{2-q}$.  Upon integration, the lower bound
\[ \int_{0}^{1} \frac{(1-s)|B|^2}{|A+sB|^{2-q}}\,ds \geq \frac{|B|^q}{2(2^{2-q})}\]
follows.
 \end{proof}


Let $u \in \sca_{\lambda}$.  Applying Lemma \ref{l2} to $A:=\lambda \1$ and $B:=\nabla u-\lambda \1$, we find that with $M:=|A|=\sqrt{2}\lambda$,
\begin{equation}\label{excess1}|\nabla u|^{q} \geq |\lambda \1|^q + q  |\lambda \1|^{q-2}(\nabla u - \lambda \1) \cdot \lambda \1 +F_{M}(\nabla u - \lambda \1) \end{equation}   
where the function $F_{M}: \R^{2 \times 2} \to \R$ is defined by  
\begin{displaymath}F_{M}(B) := \left\{\begin{array}{l l} C_{1}(M,q) |B|^2 & \textrm{if}\; |B| \leq M, \\
C_{2}(q) |B|^q & \textrm{if}\; |B| \geq M. \end{array}\right.
\end{displaymath}
Now
\[ |\nabla u - \lambda \1| \geq \dist(\nabla u, \lambda  SO(2)) \]
and since, by polar factorization, 
\begin{eqnarray*}
\dist(\nabla u,  \lambda  SO(2)) =|\sqrt{\nabla u^T \nabla u}-\l\1|=|(\lambda_{1}(\nabla u), \lambda_{2}(\nabla u)) - (\lambda,\lambda)|,
\end{eqnarray*}
where $0<\lambda_{1}(\nabla u) \leq \lambda_{2}(\nabla u)$ are the singular values of $\nabla u$, we have 
\begin{equation}\label{lamlam}
|\nabla u - \lambda \1| \geq |\Lambda-\Lambda_{0}|.
\end{equation}
Here, $\Lambda := (\lambda_{1},\lambda_{2})$, where we leave out the dependence on $\nabla u$ for clarity, and $\Lambda_{0}:=(\lambda,\lambda)$.  Next, define $f_{M}: \R^+ \to \R^+$ by 
\begin{equation}\label{deflittlef}f_{M}(t) := \min\{C_{1}(M,q)t^2, C_{2}(q)t^q\}, \end{equation}
where $C_{1}(M,q)$ and $C_{2}(q)$ are as in \eqref{c1} and \eqref{c2}, respectively.
\begin{remark}\label{cont1}\rm{We note that $f_{M}$ is continuous on $\R^+$ and  $C_{1}(M,q)t^2=C_{2}(q)t^q$ if and only if $t=M$.  Thus the growth of $f_{M}$ switches from quadratic on $[0,M]$ to $q$-growth on $[M,+\infty)$.  We remark that the continuity is a consequence of the improved (i.e.\ increased) value for $C_{2}(q)$ provided in Lemma \ref{l2}.  More importantly, a larger value for $C_{2}(q)$ makes our estimate of the critical load more accurate:  see \eqref{suff2}, for example. }\end{remark}

Then, by combining \eqref{lamlam} and \eqref{deflittlef} with the definition of $F_{M}$, we obtain
\[ F_{M}(\nabla u -\lambda \1) \geq f_{M}(|\Lambda-\Lambda_{0}|).\]
Therefore, by \eqref{excess1}, 
\[|\nabla u|^{q} \geq |\lambda \1|^q + q  |\lambda \1|^{q-2}(\nabla u - \lambda \1) \cdot \lambda \1 +f_{\sqrt 2 \l}(|\Lambda-\Lambda_{0}|). \] 
Integrating this, applying the definition of the stored-energy function $W$, using 
\[
\int_{\om} (\nabla u -\lambda \1) \, dx = 0,
\] 
and recalling that $\det \nabla u=\l_1\l_2$, gives 
\begin{equation}\label{estampie} 
I(u) \geq \int_{\om} \big(|\lambda \1|^{q} + f_{\sqrt 2 \l}(|\Lambda-\Lambda_{0}|) +h(\lambda_{1}\lambda_{2})\big)\,dx.
\end{equation}
Then in view of the convexity of $h$ we get
\begin{eqnarray*}
I(u)-I(u_\lambda) &\geq& \int_{\om} f_{\sqrt{2}\lambda}(|\Lambda-\Lambda_{0}|)\dx + \int_{\om} \big(h(\lambda_{1}\lambda_{2})-h(\lambda^2)\big)\dx
\\
&\geq & \int_{\om} f_{\sqrt{2}\lambda}(|\Lambda-\Lambda_{0}|)\dx +h'(\l^2) \int_{\om} \big(\lambda_{1}\lambda_{2}-\lambda^2\big)\dx.
\end{eqnarray*}
 As has already observed, we need only consider $h'(\l^2)>0$, since Proposition \ref{p1} covers 
the case $h'(\l^2)\leq 0$.

Note that  
\begin{equation*}f_{\sqrt{2}\lambda}(|\Lambda-\Lambda_{0}|) +  h'(\l^2)(\lambda_{1}\lambda_{2}-\lambda^2) = 
\scg_{1}^\l(\Lambda) + \scg_{2}^\l(\Lambda), 
\end{equation*}
where 
\begin{equation}\label{j:g1}\scg_{1}^\l(\Lambda):=f_{\sqrt{2}\lambda}(|\Lambda-\Lambda_{0}|) +  h'(\l^2)(\lambda_{1}-\lambda)(\lambda_{2}-\lambda)\end{equation}
and
\begin{equation}\label{j:g2}\scg_{2}^\l(\Lambda):=\l h'(\l^2)(\lambda_{1}+\lambda_{2}-2\lambda),\end{equation}
so that we have
\[
I(u)-I(u_\l) \geq \int_{\om} \scg_{1}^\l(\Lambda)\,dx +\int_{\om} \scg_{2}^\l(\Lambda)\,dx.
\]

\noindent 
The rest of this section is devoted to finding conditions  on $\l$  which ensure that 
\[
\int_{\om} \scg_{i}^\l(\Lambda)\,dx \geq 0\quad  \text{for $i=1,2$}.
\]
The following result, in which inequality \eqref{qctrace} is part of \cite[Lemma 5.3]{Ba77}, allows us to deal with the term involving $\scg_{2}^\l$. We give a short elementary proof here to keep the paper self-contained; we also give a refined version of the estimate \eqref{qctrace} which provides an `excess term' (an estimate of the difference between the two sides of the inequality \eqref{qctrace}): see \eqref{c:excess} below.

\begin{lemma}\label{l3} 
Let $u \in W^{1,1}(\om,\R^2)$ satisfy $u=u_{\lambda}$ on $\partial \om$ and suppose that $\det \nabla u > 0$ a.e.\ in $\om$.
Then 
\begin{equation}\label{qctrace}
\int_{\om} \big(\lambda_{1}+\lambda_{2}\big)\,dx \geq 2\l \, \scl^2(\om),
\end{equation}
where $0< \l_1\leq \l_2$ denote the singular values of $\nabla u$.   Moreover,
\begin{equation}\label{c:excess}\int_{\om}  \big(\lambda_{1}+\lambda_{2}- 2\l\big) \,dx \geq \int_{\om} \psi(u,\lambda)\,dx, \end{equation}
where
\[\psi(u,\lambda) :=\frac{2\lambda^2(\curl u)^2}{3((\curl u)^2 + \max\{4\lambda^2, (\Div u)^2\})^{\frac{3}{2}}.}\]
\end{lemma}
\begin{proof}   We first give a direct proof of \eqref{qctrace}. 

The singular value decomposition theorem (see e.g. \cite[Theorem 13.3]{Dac08}) yields
\[ \nabla u = R D(\lambda_{1},\lambda_{2}) Q, \]
where $R, Q \in O(2)$ and 
\[
D(\lambda_{1},\lambda_{2}):=\left( \begin{array}{l l } \l_1 & 0 \\ 0 & \l_2 \end{array}\right).
\] 
Hence
\[ \tr \nabla u = \tr (QR D(\lambda_{1},\lambda_{2})).\]
Since $QR \in O(2)$, it must be of the form 
\begin{displaymath}
QR = \left( \begin{array}{l l } \cos \sigma & \pm \sin \sigma \\ \sin \sigma & \mp \cos \sigma \end{array}\right),
\end{displaymath}   
therefore 
\[ \tr \nabla u = \cos \sigma (\lambda_{1} \mp \lambda_{2}).\]
It can now be checked that
\[ \tr \nabla u \leq \lambda_{1} + \lambda_{2}.\]
Then integrating the latter expression over $\om$ and using the fact that the weak derivative satisfies
\[ \int_{\om} \tr \nabla u  \, dx = \int_{\om} \tr \nabla u_{\lambda}  \, dx = 2\l \,\scl^2(\om) \]
yields \eqref{qctrace}.

To prove \eqref{c:excess}, let $\xi\in \R^{2\times2}$, denote by $\l_1(\xi), \l_2(\xi)$ the singular values of $\xi$ and define the function $\varphi: \R^{2 \times 2} \to [0,+\infty)$ by 
\begin{equation}\label{j:defphi} \varphi(\xi) := \l_1(\xi)+\l_2(\xi).\end{equation}
Notice that
\begin{equation}
\label{j:expforphi}\varphi(\xi) =\sqrt{|\xi|^2+2 \det \xi}.
\end{equation}
Then 
by applying the standard identity 
\[g(1)=g(0)+g'(0)+\int_{0}^{1}(1-s)g''(s)\,ds\]
to the function $g(s):=\varphi((1-s)\lambda \1+ s \xi)$ defined for $s\in[0,1]$,
we obtain
\begin{align}\label{excess2}
\varphi(\xi) &= \varphi(\lambda \1) + \tr(\xi-\lam \1) + {}\\\nonumber &+
 \int_{0}^{1}(1-s)\frac{\varphi^{2}(\omega(s))\varphi^{2}(\xi-\lambda \1)-((\omega(s)+\cof \omega(s))\cdot (\xi-\lambda \1))^{2}}{\varphi^{3}(\omega(s))}\,ds,\end{align}
where
\[ \omega(s) := (1-s)\lambda \1 + s \xi \quad \text{for} \quad 0 \leq s \leq 1.\]
For later use we note that the term 
\[X(\omega(s),\xi-\lambda \1):=\frac{\varphi^{2}(\omega(s))\varphi^{2}(\xi-\lambda \1)-((\omega(s)+\cof \omega(s))\cdot (\xi-\lambda \1))^{2}}{\varphi^{3}(\omega(s))}\]
can be rewritten as
\begin{equation}\label{antitr}X(\omega(s),\xi-\lam \1) = \frac{({\atr}(\omega(s)) \tr(\xi-\lam \1) - \atr(\xi-\lam \1) \tr(\omega(s)))^{2}}{\varphi^{3}(\omega(s))}.\end{equation}
Here, $\atr(\eta)$ denotes the antitrace of any $\eta\in \R^{2\times 2}$ and is defined by $\atr(\eta) := \eta_{12}-\eta_{21}$.   Note that, thanks to \eqref{antitr}, $X(\cdot,\cdot) \geq 0$ for all $\xi$ and $s\in [0,1]$, so that by letting $\xi=\nabla u$ in \eqref{excess2} we obtain an alternative proof of \eqref{qctrace}.  

Then \eqref{c:excess} follows by calculating the terms in \eqref{antitr}.  Letting $\xi = \nabla u$ again, we have $\omega(s) = \lam \1 + s(\nabla u - \lam \1)$, and 
\begin{align*} {\atr} (\nabla u - \lambda \1) & = \curl u \\ 
{\tr}(\nabla u - \lam \1) &= \Div u - 2 \lam \\
{\atr}(\omega(s)) & = s \,\curl u \\
{\tr}(\omega(s)) & = s \, \Div u+(1-s)2\lam. 
\end{align*}
This gives
\begin{equation}\label{lb1}X(\omega(s),\xi-\lambda \1)=\frac{4\lam^2 (\curl u)^2}{\varphi^{3}(\lam \1+s(\nabla u -\lam \1))}.\end{equation}
Now 
\[\varphi^2(\eta) = (\atr(\eta))^2+(\tr(\eta))^{2},\]
so we have
\[ \varphi^{2}(\lam \1+s(\nabla u -\lam \1)) = s^2 (\curl u)^2 + (s\,\Div u + 2 (1-s)\lam)^2.\]
Since the function 
\[p:s \mapsto (s\,\Div u + 2 (1-s)\lam)^2\]
is convex, its maximum on the interval $[0,1]$ must be $\max\{p(0),p(1)\}$.  Hence
\begin{align*}  \varphi^{2}(\lam \1+s(\nabla u -\lam \1)) & \leq  (\curl u)^2 +  \max\{4\lambda^2, (\Div u)^2\}
\end{align*}
uniformly in $s$.  Therefore \eqref{lb1} gives
\[
X(\omega(s),\xi-\lambda \1) \geq \frac{4\lam^2 (\curl u)^2}{((\curl u)^2 +  \max\{4\lambda^2, (\Div u)^2\})^{\frac{3}{2}}}.
\]
Inserting this into \eqref{excess2}, recalling that
\[
\l_1+\l_2-2\l=\varphi(\nabla u)-\varphi(\l\1),
\]
and carrying out what becomes a trivial integration yields \eqref{c:excess}.
\end{proof}

We now return to the estimate of $\scg_{2}^\l$.  Indeed,  since we are working under the assumption $\l h'(\l^2)>0$ for every $\l>0$, applying Lemma \ref{l3} gives 
\begin{equation}\label{c:G2}
\int_{\om} \scg_{2}^\l(\Lambda)\,dx \geq 0, 
\end{equation}
as desired.

To deal with the term involving $\scg_{1}^\l$ we find an explicit condition on $\l$ which ensures that 
$\scg_{1}^\l(\Lambda) \geq 0$ 
holds pointwise for $\Lambda \in  \R^{++}$
where 
\[
\R^{++}:=\{x\in \R^2: \ x_{1}, x_{2} > 0\}.
\]
\begin{lemma}\label{lesperanza} The function
\[\mathcal{G}_{1}^\l(\Lambda)=f_{\sqrt{2}\lambda}(|\Lambda-\Lambda_{0}|) + h'(\l^2)(\lambda_{1}-\lambda)(\lambda_{2}-\lambda)\]
is pointwise nonnegative on $\R^{++}$
provided
\begin{equation}\label{first} 
C_{1}(\sqrt{2}\lambda,q) \geq  h'(\l^2)/2, 
\end{equation}
and
\begin{equation}\label{second} 
\frac{C_{2}(q)}{ h'(\l^2)\lambda^{2-q}} \geq  (q-1)^{(q-1)/2}q^{-q/2}.
\end{equation}
Moreover, inequality \eqref{second} implies \eqref{first}.
\end{lemma}
\begin{proof}  We divide the proof into two parts, the first of which is devoted to proving the sufficiency of \eqref{first} and \eqref{second}.

\vspace{2mm}
\noindent \textbf{Part 1.}  To shorten notation set  $Y:=h'(\l^2)$. Let $\Lambda-\Lambda_0=(\l_1-\l,\l_2-\l)=(\rho \cos \mu, \rho \sin \mu)$ and let 
$C_{1}:=C_{1}(\sqrt{2}\lambda,q)$ and $C_{2}:=C_{2}(q)$, as defined in \eqref{c1} and \eqref{c2} respectively.   Let $G(\rho,\mu):=\scg_{1}^\l(\Lambda)$ and note that (using \eqref{deflittlef} with $M=\sqrt{2}\lambda$)
\begin{gather}\label{G}
G(\rho,\mu) = \left\{\begin{array}{ll}C_{1}\rho^2+Y\rho^2 \sin \mu \cos \mu & \text{if} \ \rho \leq \sqrt{2}\lam \\
C_{2}\rho^q + Y\rho^2 \sin \mu \cos \mu & \text{if} \ \rho \geq \sqrt{2}\lam.                        
\end{array}\right.   
\end{gather} 
Firstly, if $\rho \leq \sqrt{2}\lam$ then $G(\rho, \mu) \geq 0$ if and only if $C_{1}+Y\sin \mu \cos \mu \geq 0$ for all $\mu$.  Whence $C_{1} -Y/2 \geq 0$, which is \eqref{first}.  We henceforth suppose that $\eqref{first}$ holds.

Inequality \eqref{second} essentially prevents $G(\rho,\mu)$ from vanishing 
\emph{outside} the set $B(\Lambda_{0},\sqrt{2}\lam) \cap \R^{++}$.  By symmetry, we need only consider $\mu \in [-\pi/4,\pi/4]$, and since $G(\rho,\mu)\geq 0$ if $0 \leq \mu \leq \pi/4$, we can restrict attention to $-\pi/4 < \mu \leq 0$.  Moreover, since $G(\rho,0)$ is obviously nonegative, we can also exclude $\mu=0$.   Now, in view of \eqref{first}, the only way $G(\rho,\mu)$ can vanish is if $\rho \geq \sqrt{2}\lam$.  In the region $\rho \geq \sqrt{2}\lam$,  $-\pi/4 < \mu < 0$
\[
G(\rho,\mu)=C_{2}\rho^q-Y|\sin \mu \cos \mu|\rho^2,
\]
and since $1 < q < 2$, it must be that $G(\rho,\mu) < 0$ for sufficiently large $\rho$ and each fixed $\mu$.  Also, since $G(\rho,\mu)$ is continuous and since, by \eqref{first}, $G(\sqrt{2}\lam,\mu) \geq 0$, it follows that
\[ \rb(\mu) := \inf\{\rho \geq \sqrt{2}\lam: \ C_{2}\rho^q-Y|\sin \mu \cos \mu|\rho^2=0\}\]
is well-defined.  Thus $\rb(\mu)$ satisfies 
\begin{equation}
 \label{rowbar} C_{2}\rb(\mu)^q-Y|\sin\mu\cos\mu|\rb(\mu)^2 = 0.
\end{equation}
Now, if the point $(\rb(\mu) \cos \mu + \lam, \rb(\mu) \sin \mu + \lam)$ lies in the interior of $\R^{++}$ then, by making $\rho$ slightly larger, we ensure $G(\rho,\mu)<0$.   Since $-\pi/4<\mu < 0$, the inclusion 
\[(\rb(\mu) \cos \mu + \lam, \rb(\mu) \sin \mu + \lam) \in \R^{++}\]
is prevented when and only when 
\begin{equation}\label{barast} \rb(\mu) \geq \rho^{\ast}(\mu),\end{equation}
where $\rast(\mu)$ satisfies $\rast(\mu) \sin \mu + \lambda = 0$ and $-\pi/4<\mu<0$.

Using \eqref{rowbar} and the definition of $\rast$, inequality \eqref{barast} is equivalent to 
\begin{equation}\label{emu} \frac{C_{2}}{Y \lam^{2-q}} \geq \underbrace{\cos \mu |\sin \mu|^{q-1}}_{=:e(\mu)},\end{equation}
 where $-\pi/4 < \mu < 0$.     It can be checked that 
\begin{equation}\label{emumax}
\max_{(-\pi/4,0)}e = (q-1)^{(q-1)/2}q^{-q/2},
\end{equation}
the maximum occurring at $\mu$ such that $\cos^{2}\mu=1/q$.  Inequality \eqref{second} now follows.

\vspace{2mm}
\noindent \textbf{Part 2} \ \  We prove that \eqref{second} implies \eqref{first}.  First note that dividing both sides of \eqref{second} by $2^{(2-q)/2}$ gives
\begin{equation}\label{secimpfir} \frac{C_{1}(\sqrt{2}\lam,q)}{Y}  \geq \underbrace{\left(\frac{(q-1)^{q-1}q^{-q}}{2^{2-q}}\right)^{1/2}}_{=:y(q)}.\end{equation}
Let $\gamma(q)=2 \ln y(q)$ and calculate $\gamma'(q)=\ln\left(2\left(1-\frac{1}{q}\right)\right)$.  Now $1<q<2$, so $2\left(1-\frac{1}{q}\right) \in (0,1)$, and hence $\gamma'(q) < 0$ on $(1,2)$.  It follows that $y$ is a decreasing function of $q$ on $(1,2)$, and since $y(q) \to \frac{1}{2}$ as $q \to 2-$, the right-hand side of \eqref{secimpfir} is bounded below by $\frac{1}{2}$.  Hence \eqref{first} holds.
\end{proof}

We now draw the preceding discussions and results together.

\begin{theorem}\label{suffthm1} Let the stored energy function $W: \R^{2 \times 2} \to [0,+\infty]$ be given by 
\begin{equation*}
W(A) :=  |A|^{q}+ h(\det A),
\end{equation*}
where $1< q < 2$ and $h\colon \R \to [0,+\infty]$ satisfies $(H1)-(H3)$.  Let $\l>0$ be such that
\begin{equation}\label{suff2}\frac{1}{2^{3-q}h'(\l^2)\lambda^{2-q}} \geq  (q-1)^{(q-1)/2}q^{-q/2}.\end{equation}
Then any $u \in \sca_{\lambda}$ satisfies $I(u) \geq I(\ul)$.
\end{theorem}

\subsection{Error estimates}\label{s2:error}

In this section we are interested in understanding the properties of those $u \in \mathcal A_\l$ such that $I(u)-I(\ul)$ is small and positive.   Hence we focus on the case $h'(\l^2)>0$  to which the results of the previous section apply.  Accordingly, we impose the hypotheses of Theorem \ref{suffthm1} and strengthen inequality \eqref{suff2} to read
\begin{equation}\label{suff3}\frac{1}{2^{3-q}h'(\l^2) \lambda^{2-q}} >  (q-1)^{(q-1)/2}q^{-q/2}.\end{equation}
The main result of this subsection is the following.
\begin{theorem}\label{suffthm2} 
 Assume that \eqref{suff3} holds. Then there is a constant $c=c(\om,\l,q)>0$ such that for every $u \in \sca_{\lam}$ 
\begin{equation}\label{rig1}\int_{\om}\min\{|\nabla u - \lambda \1|^2,|\nabla u - \lambda \1|^q\}\,dx \leq c\, \delta(u),
\end{equation}
where $\delta(u):=I(u)-I(u_{\lam})$.  Moreover,
\begin{equation}\label{rig2}
\l\, h'(\l^2) \int_{\om} \frac{2\lambda^2(\curl u)^2}{3((\curl u)^2 + \max\{4\lambda^2, (\Div u)^2\})^{\frac{3}{2}}}\,dx \leq \delta(u).\end{equation}
\end{theorem}

\noindent The proof of Theorem \ref{suffthm2} is given in stages below.  In view of
\begin{equation}\label{c:error}
\int_{\om}\scg_{1}^\l(\Lambda) \,dx+\int_{\om}\scg_{2}^\l(\Lambda) \,dx \leq \delta(u),
\end{equation}
the idea is that if $\delta(u)$ is small then the same must be true of the two (necessarily nonnegative) terms in the right-hand side of \eqref{c:error}.
The first inequality, \eqref{rig1}, follows from a smallness assumption on $\int_{\om}\scg_{1}^\l(\Lambda)\dx$: see Proposition \ref{rig1true} below, while inequality \eqref{rig2} is a consequence of small $\int_{\om} \scg_{2}^\l(\Lambda)\dx$ and follows in a straightforward way from \eqref{c:excess}. 

We remark that an inequality like \eqref{rig2} is not available in the three dimensional case, or at least we could not derive it.  The chief difficulty is the lack of an explicit expression for $\l_{1}(\xi) + \l_{2}(\xi)+\l_{3}(\xi)$ for $\xi\in \R^{3\times 3}$:  cf. \eqref{j:defphi} and \eqref{j:expforphi}.

We now turn to inequality \eqref{rig1}. To this end we introduce the function $g\colon [0,+\infty)\to [0,+\infty)$ defined by
\begin{equation}\label{c:gi}
g(t):=\begin{cases}
\ds\frac{t^2}{2} & \text{if }\; 0\leq t \leq 1,
\cr\cr 
\ds\frac{t^q}{q} + \frac{1}{2}-\frac{1}{q} & \text{if }\; t \geq 1.
\end{cases}
\end{equation} 
For later use we notice that $g$ is convex.
\begin{lemma}\label{cdmwarmup}  Let \eqref{suff3} hold.  Then there is a constant $c_0=c_0(\l,q)>0$ such that 
\begin{equation}\label{LBF}
\mathcal{G}_{1}^\l(\Lambda) \geq c_0\, g(|\Lambda-\Lambda_{0}|) \quad \text{on $\R^{++}$}
\end{equation}
where $g$ is as in \eqref{c:gi}.
\end{lemma}
\begin{proof}It is clear from the last part of the proof of Lemma \ref{lesperanza} that inequality \eqref{suff3} implies that \eqref{first} holds with strict inequality.   Thus
\begin{equation}\label{inner}\scg_{1}^\l(\Lambda) \geq c|\Lambda-\Lambda_{0}|^{2} \quad \quad\rm{if} \ |\Lambda- \Lambda_{0}| \leq \sqrt{2}\lam\end{equation}
for some constant $c>0$.   

Reusing the notation $\Lambda-\Lambda_{0}= \rho (\cos \mu, \sin \mu)$ and $G(\rho,\mu):=\mathcal{G}_{1}^\l(\Lambda)$, the case $\rho \geq \sqrt{2}\lam$ can be handled as follows.  Let $\eps>0$ and write 
\begin{align*} G(\rho,\mu) & = C_{2}\rho^q - Y|\sin \mu \cos \mu| \rho^2 \\
& = (C_{2}-\eps)\rho^q -Y|\sin \mu \cos \mu|\rho^2 + \eps\rho^q,
\end{align*}
where $Y:=h'(\l^2)$.
By applying the reasoning in the proof of Lemma \ref{lesperanza} to the function 
\[\tilde{G}(\rho,\mu): = (C_{2}-\eps)\rho^q -Y|\sin \mu \cos \mu|\rho^2,\]
we see that $\tilde{G}(\rho,\mu)\geq 0$ provided 
\begin{equation}\label{emuplus}
\frac{C_{2}-\eps}{Y \lam^{2-q}} \geq  (q-1)^{(q-1)/2}q^{-q/2}.
\end{equation}
Inequality \eqref{suff3} clearly implies that $C_{2}$ exceeds the right-hand side of \eqref{emuplus} by a fixed amount; thus, if $\eps>0$ is sufficiently small, inequality \eqref{emuplus} holds.  
Hence 
\begin{equation}\label{outer}\scg_{1}^\l(\Lambda) \geq \eps|\Lambda-\Lambda_{0}|^{q} \quad \quad\rm{if} \ |\Lambda- \Lambda_{0}| \geq \sqrt{2}\lam.\end{equation}
Inequalities \eqref{inner} and \eqref{outer} are easily combined to give \eqref{LBF}.
\end{proof}

We will see that inequality \eqref{rig1} is a consequence of the $L^2+L^q$ rigidity estimate \cite[Theorem 1.1]{CDM12}. We recall here the following variant (see \cite[Lemma 3.1]{ADD}) which is suitable for our purposes.

\begin{lemma}\label{c:ADD}
Let $U\subset \R^n$ be a bounded domain with Lipschitz boundary. Let $\l>0$ and $g$ be as in \eqref{c:gi}. There exists a constant $c=c(U,\l,q)>0$
with the following property: for every $v\in W^{1,q}(U;\R^n)$ there is a constant rotation $R\in SO(n)$ satisfying 
\[
\int_U g(|\nabla v-\l R|)\dx \leq c \int_U g({\rm dist}(\nabla v, \l SO(n)))\dx.
\]   
\end{lemma}
\begin{proof}
Once we observe that, thanks to \cite[Theorem 3.1]{FJM02} we can find $c=c(U)>0$ such that for every 
$w\in W^{1,2}(U;\R^n)$ there is a constant rotation $R\in SO(n)$ satisfying
\[
\int_U |\nabla w-\l R|^2\dx \leq c \int_U {\rm dist}^2(\nabla w, \l SO(n))\dx,
\]
the proof then closely follows that of \cite[Lemma 3.1]{ADD}.
\end{proof}

\begin{proposition}\label{rig1true}
There is a constant $c=c(\om,\l,q)>0$ such that
\begin{equation}\label{notsoeasy}
\int_{\om}\min\{|\nabla u- \l\1|^2,|\nabla u - \l \1|^q\}\,dx \leq c \delta(u).
\end{equation}
\end{proposition}
\begin{proof} Throughout this proof $c$ denotes a generic strictly positive constant possibly depending on $\om$, $\l$, and $q$.   By \eqref{c:G2} and \eqref{c:error} we have
\[
\int_{\om}\scg_{1}^\l(\Lambda) \,dx\leq\delta(u). 
\]
Hence on recalling that  
\[ 
|\Lambda - \Lambda_{0}|=\dist (\nabla u, \lambda SO(2)),
\]  
and by appealing to Lemma \ref{cdmwarmup}, we get
\begin{equation*}
 c_0 \int_\om g(\dist(\nabla u,\lambda SO(2))\dx \leq \delta(u).
\end{equation*}
Then Lemma \ref{c:ADD} provides us with $c>0$ and $R\in SO(2)$
such that
\begin{equation}\label{c:rigidity}
 \int_\om g(|\nabla u -\l R|)\dx\leq c\,\delta(u).
\end{equation} 
We claim that
\begin{equation}\label{c:claim}
|\1-R|^2\leq c\, \delta(u). 
\end{equation}
By virtue of the convexity of $g$, combining Jensen's inequality with \eqref{c:rigidity} gives 
\begin{equation}\label{c:step00}
g\left(\frac{1}{\scl^2(\om)}\int_\om |\nabla u -\l R|\dx\right) \leq c\,\delta(u).
\end{equation}
Set $\tilde u:= u/\lambda$ and $\tilde z:=\frac{1}{\scl^2(\om)}\int_\om(\tilde u-Rx)\dx$.  Then by Poincar\'{e}'s inequality together with the continuity of the trace operator we obtain
\[
\int_{\partial \om}|\tilde u -Rx -\tilde z|\,d\mathcal H^1 \leq c \int_\om|\nabla \tilde u-R|\dx, 
\]
and hence, since $\tilde u=x$ on $\partial \om$, we deduce that
\begin{equation}\label{c:f1}
\int_{\partial \om}|(\1-R)x -\tilde z|\,d\mathcal H^1 \leq c \int_\om|\nabla \tilde u-R|\dx.
\end{equation}
Arguing as in the proof of \cite[Lemma 3.3]{ADD}, we apply \cite[Lemma 3.2]{ADD} to deduce that there exists a universal constant $\sigma>0$
such that
\begin{equation}\label{c:f2}
|\1-R|\leq \sigma \min_{z\in \R^2}\int_{\partial \om}|(\1-R)x -z|\,d\mathcal H^1.
\end{equation}
Combining \eqref{c:f1} and \eqref{c:f2} gives
\begin{eqnarray*}
|\1-R| &\leq& c \int_\om |\nabla \tilde u-R|\dx 
\\
&=&\frac{c}{\l} \int_\om |\nabla u-\l R|\dx,
\end{eqnarray*}
and therefore
\begin{equation}\label{c:square}
|\1-R|^2 \leq c \left(\frac{1}{\scl^2(\om)}\int_\om |\nabla u-\l R|\dx\right)^2.
\end{equation}
Then to prove \eqref{c:claim} we need to distinguish two cases.

\smallskip

(i) $\ds \int_\om |\nabla u-\l R|\dx \leq \scl^2(\om)$. 

\smallskip

\noindent By definition $g(t)={t^2}/{2}$ for $t\leq 1$,
so that \eqref{c:step00} and \eqref{c:square} immediately yield
\begin{equation*}
|\1-R|^2 \leq c\, g\left(\frac{1}{\scl^2(\om)}\int_\om |\nabla u-\l R|\dx\right)\leq c\,\delta(u).
\end{equation*}
\smallskip

(ii) $\ds \int_\om |\nabla u-\l R|\dx > \scl^2(\om)$\,.

\smallskip

\noindent When $t>1$ we have $g(t)>{1}/{2}$, then
\begin{eqnarray*}
|\1-R|^2 &\leq& 2(|\1|^2+|R|^2)
\\ 
&<& c\, g\left(\frac{1}{\scl^2(\om)}\int_\om |\nabla u-\l R|\dx\right)
\\
&\leq& c\,\delta(u),
\end{eqnarray*}
hence the claim is proved.

We now notice that the convexity of $g$ together with its definition entails 
\[
g(s+t)\leq c\Big(g(s)+t^2\Big)\quad \text{for every}\quad s,t\geq 0
\]
and for some $c>0$.  Indeed we have
\begin{eqnarray*}
g(s+t)&\leq & 2^q\, g\Big(\frac{s+t}{2}\Big)\\
&\leq & 2^{q-1} \big(g(s)+g(t)\big)\\
&\leq &2^{q-1} \bigg(g(s)+\frac{t^2}{q}\bigg).
\end{eqnarray*}
Then choosing $R$ as in \eqref{c:rigidity} and combining the latter with \eqref{c:claim} implies 
\begin{eqnarray}\nonumber
\int_\om g(|\nabla u-\l\1|)\dx &=&\int_\om g(|\nabla u-\l R+ \l R-\l\1|)\dx
\\\nonumber
&\leq& c\left(\int_\om g(|\nabla u-\l R|)\dx+\l^2\,|\1-R|^2\right) 
\\\label{c:per-rm}
&\leq& c\,\delta(u).
\end{eqnarray}
Finally, since we can find $c>0$ such that
\[
\min\{t^2,t^q\} \leq c\,g(t)\quad \text{for every}\quad t\geq 0,
\] 
we obtain
\[
\int_{\om}\min\{|\nabla u - \l\1|^2,|\nabla u - \l\1|^q\}\,dx \leq c \delta(u),
\]
which is the thesis.
\end{proof}
\begin{remark}\label{c:trivial}
\rm{Using \eqref{c:per-rm} and the definition of $g$ we obtain
\begin{equation}\label{c:rm1}
\int_{|\nabla u-\l\1|\leq 1}|\nabla u-\l\1|^2\dx \leq c \int_\om g(|\nabla u-\l\1|)\dx \leq c\,\delta(u).
\end{equation}
Then recalling that $q<2$, H\"older's inequality combined with \eqref{c:rm1} yields 
\begin{equation}\label{c:rm2}
\int_{|\nabla u-\l\1|\leq 1}|\nabla u-\l\1|^q\dx \leq \scl^2(\om)^{1-\frac q2} \left(\int_{|\nabla u-\l\1|\leq 1}|\nabla u-\l\1|^2\dx\right)^\frac{q}{2} \leq c\,\delta(u)^\frac{q}{2}. 
\end{equation}
On the other hand we clearly have
\begin{equation}\label{c:rm3}
\int_{|\nabla u-\l\1|>1}|\nabla u-\l\1|^q\dx \leq c\, \int_\om g(|\nabla u-\l\1|)\dx \leq c\,\delta(u).
\end{equation}
Therfore \eqref{c:rm2} and \eqref{c:rm3} together give
\[
\int_\om |\nabla u-\l\1|^q\dx \leq c\, \Big(\delta(u)^{\frac{q}{2}}+\delta(u)\Big),
\]
which on applying Poincar\'{e}'s inequality finally implies
\begin{equation}\label{c:well}
\|u-u_\l\|_{W^{1,q}(\om;\R^2)}^q \leq c\, \Big(\delta(u)^{\frac{q}{2}}+\delta(u)\Big).
\end{equation}

\noindent If $\l$ satisfies \eqref{suff3} then from \eqref{c:well} we can conclude that $u_\l$ is the 
unique global minimiser of $I$ among all maps $u$ in $\mathcal A_\l$ and, moreover, that $u_\l$
lies in a potential well.  
}
\end{remark}

\section{The three dimensional case}\label{s3}
\noindent In this section we seek conditions analagous to those obtained in the two dimensional case ensuring that $\ul$ is the unique global minimizer of an appropriately defined stored-energy function.   For simplicity we focus on the following $W: \R^{3 \times 3} \to [0,+\infty]$ given by 
\begin{equation}\label{w3d}W(A):=|A|^{q}+\gamma|A|^{2} + Z(\cof A)+h(\det A),\end{equation}
where $2<q<3$, $\gamma>0$ is a fixed constant, $Z: \R^{3 \times 3} \to [0,+\infty)$ is convex and $C^1$, and $h$ has properties (H1)-(H3).

Applying \cite[Lemma A.1]{MSS96} to $A \mapsto |A|^q$ gives
\begin{equation}\label{c:q}
|\nabla u|^q \geq |\l\1|^q+q|\l\1|^{q-2} \l\1\cdot (\nabla u -\l\1)+\kappa |\nabla u-\l\1|^q,
\end{equation}
where
\begin{equation}\label{j:kappalimits}
2^{2-q} \leq \kappa \leq  q2^{1-q}  .\end{equation} 
Moreover, we clearly have
\begin{equation}\label{c:2}
\gamma |\nabla u|^2 \geq \gamma |\l\1|^2+2\gamma \l\1\cdot (\nabla u -\l\1)+\gamma |\nabla u-\l\1|^2.
\end{equation}
Therefore, by gathering \eqref{c:q} and \eqref{c:2} and appealing to the convexity of $Z$ and $h$, we obtain
\begin{eqnarray}\nonumber 
W(\nabla u) &\geq & W(\nabla u_{\lambda}) + q|\l \1|^{q-2}\l \1\cdot(\nabla u    - \l \1) + \kappa |\nabla u-\l\1|^{q}
\\\label{ineq0:3d}
&+& 2\gamma \l \1 \cdot (\nabla u    - \l \1) + \gamma |\nabla u - \l \1|^{2}
\\\nonumber
&+& D_{A}Z(\cof \l \1) \cdot (\cof \nabla u - \cof \l \1)\\ \nonumber
&+& h'(\l^{3})(\det \nabla u-\l^3), 
\end{eqnarray}
for any $u \in \mathcal A_\l$, where $\sca_{\l}$ is the class of admissible maps given by \eqref{j:defalam} with $n=3$.  
Integrating \eqref{ineq0:3d} and using the facts that both $\nabla u$ and $\cof \nabla u$ are null Lagrangians in 
$W^{1,q}(\om,\R^{3})$ for $q \geq 2$, we obtain
\begin{equation}I(u)-I(\ul) \geq  \int_{\om} \big(\kappa|\nabla u - \l \1|^q+\gamma |\nabla u - \l \1|^2 + h'(\l^3)(\det \nabla u - \l^3)\big)\,dx \label{ineq2:3d}
\end{equation}
By analogy with Proposition \ref{p1} we can deal with the case $h'(\l^{3})\leq 0$ by imposing condition (INV) on a suitably defined extension of $u$, as follows.   

\begin{proposition}\label{j:extraprop}  Suppose that $W: \R^{3 \times 3} \to [0,+\infty]$ is given by
\[W(A):=|A|^{q}+\gamma|A|^{2} + Z(\cof A)+h(\det A)\]
where $2<q<3$, $\gamma>0$ is a fixed constant, $Z: \R^{3 \times 3} \to [0,+\infty)$ is convex and $C^1$, and $h$ has properties (H1)-(H3).  Let $B(0,M)$ contain $\bar{\om}$ and denote by $u^{\textrm{e}}$ the extension of $u$ to $B(0,M) \setminus \om$ defined by
\[u^{\textrm{e}}(x) := \left\{\begin{array}{l l} u(x) & \textrm{if} \ x \in \om, \\ 
\ul(x)  & \textrm{if} \ x \in B(0,M) \setminus \om. \end{array}\right.\]
Assume that $u^{\textrm{e}}$ satisfies the hypotheses of \cite[Lemma 8.1]{MS95} in the case that $n=3$. Then if $\int_{\om} \det \nabla u \,dx= \int_{\om}\det \nabla \ul \,dx$ or if $h'(\l^{3}) \leq 0$, the inequality $I(u) \geq I(\ul)$ holds.   
\end{proposition}
 
\begin{proof}
By \eqref{ineq2:3d} it is enough to show that $h'(\l^3)\int_{\om}(\det \nabla u - \l^3)\,dx \geq 0$. 
The argument which precedes Proposition \ref{p1} implies that the integral term is not greater than zero, which when coupled with the assumption $h'(\l^3)\leq 0$ easily gives the desired inequality.
\end{proof}

Let $0<\l_{1}\leq \l_2 \leq \l_3$ be the singular values of $\nabla u$ and define the vectors $\Lambda:=(\l_{1},\l_{2},\l_{3})$ and $\Lambda_{0}:=(\l,\l,\l)$.  Recall that 
\[ 
|\nabla u - \l \1| \geq |\Lambda - \Lambda_{0}|;
\]
then \eqref{ineq2:3d} implies 
\begin{equation} \label{c:ineq2:3d}
I(u)-I(\ul) \geq  \int_{\om} \big(\kappa|\Lambda - \Lambda_{0}|^q+\gamma |\Lambda - \Lambda_{0}|^2 + h'(\l^3)(\l_1\l_2\l_3 - \l^3)\big)\,dx 
\end{equation}

\noindent The next three results are devoted to the case $h'(\l^3)>0$.

\begin{lemma} Let $W$ be as in \eqref{w3d} and let $u\in \mathcal A_\l$. 
Then \begin{equation}\label{ineq3:3d} I(u) - I(\ul) \geq \int_{\om} \big(\scf_{1}^\l(\Lambda)+\scf_{2}^\l(\Lambda)\big)\,dx, \end{equation}
where 
\[\scf_{1}^\l(\Lambda) := \kappa|\Lambda-\Lambda_{0}|^{q}+h'(\l^3)({\l}_{1}-\l)({\l}_{2}-\l)({\l}_{3}-\l)\]
and
\[
\scf_{2}^\l(\Lambda) :=  \gamma |\Lambda-\Lambda_{0}|^{2} + \l h'(\l^{3}) \sum_{i < j} ({\l}_{i}-\l)({\l}_{j}-\l).
\]
\end{lemma}

\begin{proof} For brevity we write $\hat{\l}_{i} := \l_{i} - \l$ for $i=1,2,3$. It follows that 
\begin{equation}\label{ineq4:3d} 
\det \nabla u -\l^{3} = \hat{\l}_{1}\hat{\l}_{2}\hat{\l}_{3} + \l \sum_{i<j}\hat{\l}_{i}  \hat{\l}_{j} + \l^{2}\sum_{i=1}^{3}\hat{\l}_{i}.\end{equation}
Inserting this into \eqref{c:ineq2:3d} gives
\[ I(u)-I(\ul) \geq \int_{\om}\left(\scf^\l_{1}(\Lambda)+\scf^\l_{2}(\Lambda)\right) \,dx+\l^{2}h'(\l^{3})\int_{\om} \sum_{i=1}^{3} \hat{\l}_{i} \,dx.\]
Since the last integral may be written as
\[ \int_{\om} \sum_{i=1}^{3} \hat{\l}_{i} \,dx = \int_{\om} (\l_{1}+\l_{2}+\l_{3}-3\l)\,dx,\]
we can apply \cite[Lemma 5.3]{Ba77} again to deduce that 
\[
\int_{\om} (\l_{1}+\l_{2}+\l_{3})\,dx  \geq 3\l\, \scl^3(\om).
\]
Hence since $h'(\l^{3}) > 0$, \eqref{ineq3:3d} holds.
\end{proof}

By analogy with the strategy leading to Lemma \ref{lesperanza}, we now find conditions on $\l$ in terms of $\kappa$, $\gamma$ and $q$ ensuring that
\[
\begin{cases}
\scf_{1}^\l(\Lambda) \geq 0
\cr 
\scf_{2}^\l(\Lambda) \geq 0
\end{cases} \quad \text{for every}\quad \Lambda \in R^{+++},
\]
where $R^{+++} := \{x\in \R^3\colon x_{i} > 0 \ \textrm{for} \ i=1,2,3\}$.
\begin{lemma}\label{lesperanza2}  The functions $\scf_{1}^\l(\Lambda)$ and $\scf_{2}^\l(\Lambda)$ are pointwise nonnegative on $R^{+++}$ provided
\begin{equation}\label{suffcondmain:3d} \frac{\kappa}{h'(\l^{3})\l^{3-q} } \geq (q-2)^{(q-2)/2}q^{-q/2} \end{equation}
and
\begin{equation}\label{suffgamma:3d} \frac{\gamma}{\l h'(\l^{3})} \geq \frac{1}{2}.\end{equation}
\end{lemma}

\begin{proof}    In the following we let $Y:=h'(\l^3) > 0$ for brevity.  We write 
\[(\hat{\l}_{1},\hat{\l}_{2},\hat{\l}_{3}) = \rho(\cos \phi \sin \theta, \sin \phi \sin \theta, \cos \theta),\] where $\rho \geq 0$ and $0 \leq \theta \leq \pi$, $0 \leq \phi \leq 2\pi$.  In terms of $\rho, \theta$ and $\phi$ we have $\scf_{1}^\l(\Lambda) = F_{1}(\rho, \theta, \phi)$, where
\begin{equation}\label{eq:G} F_{1}(\rho, \theta, \phi) := \kappa \rho^{q} + \frac{Y \rho^{3}}{4} \sin 2 \phi \sin 2 \theta \sin \theta.\end{equation}
Since the singular values of $\nabla u$ are ordered as $\lambda_{1} \leq \lambda_{2} \leq \lambda_{3}$ the same applies to the $\hat{\l}_{i}$ for $i=1,2,3$; hence in particular $\hat{\l}_{1} \leq \hat{\l}_{2}$.  The latter implies $\phi \in [\pi/4,5\pi/4]$.   Now if $\sin 2 \phi \cos \theta \geq 0$ then the stated result would be immediate from \eqref{eq:G}.  Therefore we assume $\sin 2\phi \cos \theta < 0$ in what follows, which in view of the restriction $\pi/4 \leq \phi \leq  5 \pi/4$ implies either that $\phi \in [\pi/2,\pi]$ when $\cos \theta > 0$ or that $\phi \in [\pi/4,\pi/2] \cup [\pi,5 \pi/4]$ when $\cos \theta < 0$.  For later use we will let $S$ be the set of $(\theta,\phi)$ satisfying these restrictions.  

Let 
\[ \bar{\rho}(\theta, \phi) := \inf\{\rho > 0: \ F_{1}(\rho,\theta, \phi) =0\}\]
and note that $\bar{\rho}$ is well-defined because, in view of
\[ F_{1}(\rho,\theta,\phi) =  \rho^{q}\left(\kappa  - \frac{Y\rho^{3-q}}{4}|\sin 2 \phi\sin 2 \theta \sin \theta|\right),\]
where $q < 3$, there is always at least one positive solution to the equation $F_{1}(\rho,\theta,\phi)=0$.  Moreover, it is clear that $\bar{\rho}$ satisfies
\begin{equation}\label{eq:rhobar}  
\frac{4\kappa}{Y} {\bar{\rho}}^{q-3}(\theta,\phi) = |\sin 2 \phi \sin 2 \theta \sin \theta|.\end{equation}
Next, let us call $\rho^{\ast}(\theta,\phi) \geq 0$ an exit radius if 
\[\Lambda_{0} + \rho^{\ast} (\cos \theta \sin \phi, \sin \theta \sin \phi, \cos \theta) \in \partial \R^{+++}.\]
Thus $\rho^{\ast}={\rho_{i}}^{\ast} > 0$ for at least one $i$, where 
\begin{align*} \lambda +{\rho_{1}}^{\ast} \sin \theta \cos \phi & =  0, \\
\lambda + {\rho_{2}}^{\ast}\sin \theta \sin \phi & =  0, \\
\lambda + {\rho_{3}}^{\ast} \cos \theta & =  0.
\end{align*}
In order that $\scf_{1}^\l(\Lambda) \geq 0$ for $\Lambda \in \R^{+++}$ it should now be clear that $\bar{\rho}$ must exceed the largest exit radius, i.e., $\bar{\rho}(\theta,\phi)\geq \max\{{\rho_{1}}^{\ast},{\rho_{2}}^{\ast},{\rho_{3}}^{\ast}\}$ for each pair $(\theta, \phi)$ in $S$.  Rearranging this,  we obtain the following sufficient condition:
\begin{equation}\label{suffcond:3d}\frac{4 \kappa}{\l^{3-q}Y} \geq \max\{s_{1},s_{2},s_{3}\},\end{equation}
where 
\begin{align*}
s_{1} & := \sup_{(\theta,\phi) \in S_{1}}\frac{|\sin 2 \phi \sin 2 \theta \sin \theta|}{|\cos \phi \sin \theta|^{3-q}}, \\
s_{2} & := \sup_{(\theta, \phi) \in S_{2}}\frac{|\sin 2 \phi \sin 2 \theta \sin \theta|}{|\sin \phi \sin \theta|^{3-q}}
,\\
s_{3} & := \sup_{(\theta, \phi) \in S_{3}}\frac{|\sin 2 \phi \sin 2 \theta \sin \theta|}{|\cos \theta|^{3-q}}.
\end{align*}
Here, $S_{i}=\{(\theta,\phi) \in S: \ {\rho_{i}}^{\ast} > 0\}$ for $i=1,2,3$.

\vspace{2mm}
\noindent \textbf{To find $s_{1}$:}  Let 
\[m_{1}(\theta, \phi) := 4 |\sin \phi||\cos \phi|^{q-2}|\cos \theta||\sin \theta|^{q-1},\]
so that $s_1=\max_{S_1}m_1$.
Note that ${\rho_{1}}^{\ast} = -\l (\cos \phi \sin \theta)^{-1} > 0$ implies $\pi/2 < \phi \leq \pi$, which when combined with the restriction $(\theta, \phi) \in S$ implies $\phi \in [\pi/2,\pi]$ when $\cos \theta > 0$ or $\phi \in [\pi,5 \pi/4]$ when $\cos \theta < 0$.   Thus we need only consider these values of $\phi$ when maximizing $m_{1}(\theta,\phi)$ over $S_{1}$.    Define $f(\phi) := |\sin \phi||\cos \phi|^{q-2}$ and note that
\[ \max_{S_{1}} m_{1} = 4 \max_{0 \leq \theta \leq \pi} |e(\theta)| \max_{[\pi/2,5\pi/4]} f(\phi),\]
where the function $e$ is defined in \eqref{emu} and its maximum is given by \eqref{emumax}.  Thus
\[ \max_{S_{1}} m_{1} = 4 (q-1)^{(q-1)/2}q^{-q/2} \max_{[\pi/2,5\pi/4]} f(\phi).\]
A short calculation shows that $f$ is maximized when $\phi$ satisfies $\cos \phi =-\left((q-2)/(q-1)\right)^{\frac{1}{2}}$, which is only possible when $\phi$ belongs to $[\pi/2,3\pi/4]$.  (It is easy to check that $f$ is monotonic on $[\pi,5\pi/4]$ and that its maximum in this range is smaller than the maximum over the range $[\pi/2,\pi]$.) Hence 
\begin{equation}\label{ftildemax}\max_{[\pi/2,5\pi/4]} f(\phi)= (q-1)^{\frac{1}{2}}\left(\frac{q-1}{q-2}\right)^{\frac{q-2}{2}},\end{equation}
which gives
\[ \max_{S_{1}} m_{1} = 4 (q-2)^{(q-2)/2}q^{-q/2}.\]

\vspace{2mm}
\noindent \textbf{To find $s_{2}$:} We claim that $s_{2}=s_{1}$.   Let
\[m_{2}(\theta,\phi) := 4 |\sin \phi|^{q-2}|\cos \phi| |\sin \theta|^{q-1}|\cos \theta|\]
and note that $s_{2}=\max_{S_{2}}m_{2}$.   By definition, $(\theta, \phi) \in S_{2}$ are such that $\rho_{2}^{\ast} > 0$, so $\sin \phi <0$, from which (given that $(\theta, \phi) \in S$) it follows that $\pi < \phi \leq 5 \pi/4$.  We have $m_{2}(\theta,\phi) = |e(\theta)|\tilde{f}(\phi)$, where the function $e$ was defined in \eqref{emu} and 
\[\tilde{f}(\phi) = |\sin \phi|^{q-2}|\cos \phi|.\]  
It is straightforward to check that the maximum of the function $\tilde{f}$ occurs at $\phi$ such that 
$\sin \phi = -((q-2)/(q-1))^{\frac{1}{2}}$ and $\cos \phi = -(q-1)^{\frac{1}{2}}$, and that consequently $\max\tilde{f}=\max f$, where $f$ is as defined in the previous paragraph.   It follows that $s_{2}=s_{1}$.

\vspace{2mm}
\noindent \textbf{To find $s_{3}$:}  We claim $s_{3}=s_{1}$.  Let 
\[m_{3}(\theta, \phi) := 2 |\sin 2 \phi||\cos \theta|^{q-2}\sin^{2}\theta \]
so that $s_{3}=\max_{S_{3}}m_{3}$.
Define $r(\theta)= |\cos \theta|^{q-2}\sin^{2}\theta$. Note that $r$ is symmetric about $\theta=\pi/2$, so it suffices to consider just its restriction to $[0,\pi/2]$.   A short calculation shows that the maximum of $r$ occurs at $\theta$ satisfying $\sin^{2}\theta = 2/q$.  Thus 
\[ \max_{S_{3}} m_{3} = 4 (q-2)^{(q-2)/2}q^{-q/2}.\]  
Condition \eqref{suffcondmain:3d}  follows by inserting $s_{1}$ into \eqref{suffcond:3d}.

Finally, \eqref{suffgamma:3d} follows by writing $\scf^\l_{2}$ in terms of the coordinates $\Lambda=\Lambda_{0} + \rho(l_{1},l_{2},l_{3})$ where $l_{1}^{2}+l_{2}^{2}+l_{3}^{2}=1$, giving
\[ \scf_{2}^\l(\Lambda) =  \rho^{2} \left( \gamma + \l h'(\l^3) \big(l_{1}l_{2}+l_{1}l_{3}+l_{2}l_{3}\big)\right).\]
The minimum of $l_{1}l_{2}+l_{1}l_{3}+l_{2}l_{3}$ among unit vectors $(l_{1},l_{2},l_{3})$ is $-1/2$.  Hence $\scf^\l_{2}$ is pointwise nonnegative provided \eqref{suffgamma:3d} holds.
\end{proof}

The foregoing results imply a three dimensional analogue of Theorem \ref{suffthm1}:
\begin{theorem}\label{suffcondsummary:3d}
Let the stored energy function $W: \R^{3 \times 3} \to [0,+\infty]$ be given by 
\begin{equation*}
W(A):=|A|^{q}+\gamma|A|^{2} + Z(\cof A)+h(\det A),
\end{equation*}
where $2<q< 3$, $Z\colon \R^{3\times 3} \to [0,+\infty)$ is convex and $C^1$, and $h\colon \R \to [0,+\infty]$ satisfies  \rm{(H1)-(H3)}.  Let $\l>0$ be such that
\begin{equation}\label{j:suffcondpenultimate:3d}\frac{\kappa}{h'(\l^{3})\l^{3-q} } \geq (q-2)^{(q-2)/2}q^{-q/2},\end{equation}
where $\kappa$ is as per \eqref{w3d} and
\begin{equation}\label{j:suffcondultimate:3d} \frac{\gamma}{\l h'(\l^3)} \geq \frac{1}{2}.\end{equation}
Then any $u \in \sca_{\lambda}$ satisfies $I(u) \geq I(\ul)$.
\end{theorem}

Let us briefly compare the result of Theorem \ref{suffcondsummary:3d} with \cite[Theorem 4.1]{MSS96}.  The latter asserts that under suitable smoothness and convexity assumptions on $h$, a linear deformation $u(x)=Lx$, $u: \om \to \mathbb{R}^3$, is a global minimizer of $I$ provided
\[ h'(\det L) |L|^{3-q} \leq \frac{c_{1}}{\alpha}.\]
Here, $\alpha$ and $c_{1}$ are constants which arise in their careful analysis (see \cite[Section 3, Remark 2]{MSS96}). Inequalities \eqref{j:suffcondpenultimate:3d} and \eqref{j:suffcondultimate:3d}     say, in the particular case $L=\lambda \1$, that the affine map $\ul$ is a global minimizer of $I$ provided
\[ h'(\det L) |L|^{3-q} \leq \min\big\{3^{(3-q)/2}(q-2)^{(2-q)/2}q^{q/2}\kappa,\, 2(3^{(3-q)/2})\,\l^{2-q}\gamma
\big\}.\]   Thus our result mirrors that of \cite{MSS96} and it produces constants which are explicit up to the inequality \eqref{j:kappalimits} obeyed by $\kappa$.  In fact\footnote{This observation is due to Dr J. Deane, to whom the authors express their gratitude.}, $\kappa$ varies very nearly linearly as a function of $q$ on the interval $[2,3]$, the approximation 
$\kappa(q)\sim 3-q + (2-\sqrt{2})(q-2)$ being accurate to within $0.025$ for $q$ in $(2,3)$ and exact at the endpoints. 

\subsection{Error estimates}
In the three dimensional case error estimates follow an analogous pattern to those given in Section \ref{s2:error}, as we now show.   Let $\l>0$ be such that
\begin{equation}\label{c:cond-lam-3d}
\begin{cases}
\ds \frac{\kappa}{h'(\l^{3})\l^{3-q} } > (q-2)^{(q-2)/2}q^{-q/2}, \cr\cr
\ds \frac{\gamma}{\l h'(\l^{3})} \geq \frac{1}{2}.
\end{cases}
\end{equation}
\begin{theorem}\label{c:rig-3d} 
 Assume that \eqref{c:cond-lam-3d} holds. Then there is a constant $c=c(\om,\l,q)>0$ such that for every $u \in \sca_{\lam}$ 
\begin{equation}\label{rig-3d}
\int_{\om}|\nabla u - \lambda \1|^q\,dx \leq c\, \delta(u),
\end{equation}
where $\delta(u):=I(u)-I(u_{\lam})$.
\end{theorem}
\begin{proof}
Throughout this proof $c$ denotes a generic strictly positive constant possibly depending on $\om$, $\l$, and $q$.

The second inequality in \eqref{c:cond-lam-3d} ensures that
\begin{equation}\label{stima0}
\int_\om \scf_1^\l(\Lambda)\dx \leq \delta(u) \quad \text{for every}\quad u\in \mathcal A_\l,
\end{equation}
while the first (strict) inequality in \eqref{c:cond-lam-3d} yields
\begin{equation}\label{c:gc-q}
\scf_1^\l(\Lambda)\geq c|\Lambda -\Lambda_0|^q \quad \text{on $\R^{+++}$},
\end{equation}
for some $c>0$. To prove \eqref{c:gc-q} we make use of the same notation as in the proof 
of Lemma \ref{lesperanza2}. Let $\eps>0$ and observe that
\begin{eqnarray*}
\scf_1^\l(\Lambda)=F_{1}(\rho, \theta, \phi) = \kappa \rho^{q} - \frac{Y \rho^{3}}{4} |\sin 2 \phi \sin 2 \theta \sin \theta|
\\
=(\kappa -\epsilon)\rho^{q} - \frac{Y \rho^{3}}{4} |\sin 2 \phi \sin 2 \theta \sin \theta| +\eps \rho^q.
\end{eqnarray*}
By applying the reasoning in the proof of Lemma \ref{lesperanza2} to the function 
\[
\tilde F_{1}(\rho, \theta, \phi):= (\kappa -\epsilon)\rho^{q} - \frac{Y \rho^{3}}{4} |\sin 2 \phi \sin 2 \theta \sin \theta|,
\]
we see that $\tilde F_{1}\geq 0$ provided
\begin{equation}\label{c:stima}
\frac{\kappa -\eps}{\l^{q-3}Y}\geq (q-2)^{(q-2)/2}q^{-q/2}.
\end{equation}
Since $Y:=h'(\l^3)$, by virtue of the first inequality in \eqref{c:cond-lam-3d}, up to choosing $\eps>0$ sufficiently small, \eqref{c:stima}
is clearly fulfilled.  

Gathering \eqref{stima0}, \eqref{c:gc-q} and recalling that
\[
|\Lambda -\Lambda_0| = \dist(\nabla u,\l\,SO(3)),\]   
we thus obtain
\begin{equation}\label{ip-rig}
\int_\om \dist^q(\nabla u,\l\,SO(3))\dx \leq c \delta(u) \quad \text{for every}\; u\in \mathcal A_\l.
\end{equation}
Then invoking the rigidity estimate \cite[Theorem 3.1]{FJM02} we find $c=c(\om)>0$ such that for every $u\in \mathcal A_\l$ there is a constant
rotation $R\in SO(3)$ satisfying 
\begin{equation}\label{te-rig}
\int_\om |\nabla u-\l R|^q \dx \leq c \delta(u) \quad \text{for every}\; u\in \mathcal A_\l. 
\end{equation}
We now claim that
\begin{equation*}
|\1-R|^q\leq c\, \delta(u). 
\end{equation*}
Combining Jensen's inequality with \eqref{te-rig} gives 
\begin{equation}\label{c:step0}
\left(\int_\om |\nabla u -\l R|\dx\right)^q \leq c\,\delta(u).
\end{equation}
Set $\tilde u:= u/\lambda$ and $\tilde z:=\frac{1}{\scl^3(\om)}\int_\om(\tilde u-Rx)\dx$.  Then by Poincar\'{e}'s inequality together with the continuity of the trace operator we obtain
\[
\int_{\partial \om}|\tilde u -Rx -\tilde z|\,d\mathcal H^2 \leq c \int_\om|\nabla \tilde u-R|\dx, 
\]
and hence, since $\tilde u=x$ on $\partial \om$, we deduce
\begin{equation}\label{c:f1-3d}
\int_{\partial \om}|(\1-R)x -\tilde z|\,d\mathcal H^2 \leq c \int_\om|\nabla \tilde u-R|\dx.
\end{equation}
Arguing as in the proof of \cite[Lemma 3.3]{ADD}, we apply \cite[Lemma 3.2]{ADD} to deduce that there exists a universal constant $\sigma>0$
such that
\begin{equation}\label{c:f2-3d}
|\1-R|\leq \sigma \min_{z\in \R^3}\int_{\partial \om}|(\1-R)x -z|\,d\mathcal H^2.
\end{equation}
Combining \eqref{c:f1-3d} and \eqref{c:f2-3d} gives
\begin{eqnarray*}
|\1-R| &\leq& c \int_\om |\nabla \tilde u-R|\dx 
\\
&=&\frac{c}{\l} \int_\om |\nabla u-\l R|\dx,
\end{eqnarray*}
and therefore by \eqref{te-rig} we achieve 
\begin{equation}\label{c:id-3d}
|\1-R|^q \leq c \left(\int_\om |\nabla u-\l R|\dx\right)^q \leq c\, \delta(u),
\end{equation}
as claimed.

Finally, choosing $R$ as in \eqref{te-rig} and combining the latter with \eqref{c:id-3d} implies 
\begin{eqnarray}\nonumber
\int_\om |\nabla u-\l\1|^q\dx &=&\int_\om |\nabla u-\l R+ \l R-\l\1|^q\dx
\\\nonumber
&\leq& c\left(\int_\om |\nabla u-\l R|^q\dx+\l^q\,|\1-R|^q\right) 
\\\nonumber 
&\leq& c\,\delta(u),
\end{eqnarray}
which is the thesis.
\end{proof}
\begin{remark}
{\rm If $\l$ satisfies \eqref{c:cond-lam-3d}, from \eqref{rig-3d} we can conclude that also in this case $u_\l$ is the 
unique global minimiser of $I$ among all maps $u$ in $\mathcal A_\l$ and moreover that $u_\l$
lies in a potential well.  
}
\end{remark}

We end this section by remarking that condition \eqref{j:suffcondultimate:3d} can be removed from the statement of Theorem \ref{suffcondsummary:3d} if a certain conjecture holds, namely that the function 
\[ A \mapsto P(A) := \sum_{i < j} \l_{i}(A)\l_{j}(A) - \l \sum_{i=1}^{3}\l_{i}(A)\]
is quasiconvex at $\l \1$ (For $i=1,2,3$, $\l_i(A)$ denote, as usual, the singular values of $A\in \R^{3\times 3}$.)   Standard results (see, e.g., \cite[Theorem 5.39 (ii)]{Dac08}) imply that
\[ A \mapsto \sum_{i < j} \l_{i}(A)\l_{j}(A)\] is polyconvex and hence quasiconvex, but it remains to be seen whether subtracting the term $\sum_{i=1}^{3}\l_{i}(A)$ destroys the quasiconvexity at $\l \1$.  We conjecture that it does not. 

To see why the quasiconvexity of $P$ at $\l \1$ might matter, note that from \eqref{ineq4:3d} we can write
\[\det \nabla u - \l^{3} = \hat{\l}_{1}\hat{\l}_{2}\hat{\l}_{3} + \l\sum_{i < j} \hat{\l}_{i}\hat{\l}_{j} + \l^2 \sum_{i=1}^{3}\hat{\l}_{i}.\] 
Recalling that $\hat{\l}_{i} := \l_{i}-\l$ for $i=1,2,3$, where each $\l_{i}$ is as before, the quadratic and linear terms in the last line can be expanded and recast as
\[ \l\sum_{i < j} \hat{\l}_{i}\hat{\l}_{j} + \l^2 \sum_{i=1}^{3}\hat{\l}_{i} = \l \sum_{i<j}\l_{i}\l_{j}- \l^2 \sum_{i=1}^{3}\l_{i},\]
whose right-hand side we recognise as $\l P(\nabla u)$.  In summary, we have shown that
\[ \det \nabla u - \l^{3} =  \hat{\l}_{1}\hat{\l}_{2}\hat{\l}_{3}+ \l h'(\l^{3}) P(\nabla u).\]
Inserting this into \eqref{ineq2:3d} gives (on dropping the term with prefactor $\gamma$, since it will no longer be needed)
\begin{eqnarray*} I(u)-I(\ul) &\geq &\int_\om \big(\kappa |\nabla u -\l\1|^{q} + h'(\l^3)  \hat{\l}_{1}\hat{\l}_{2}\hat{\l}_{3} \big)\,dx + \l h'(\l^{3})\int_{\om}  P(\nabla u) \,dx\\
& = & \int_{\om} \scf_{1}^\l(\Lambda)\,dx + \l h'(\l^{3}) \int_{\om}  P(\nabla u) \,dx.
\end{eqnarray*}
If $P$ were quasiconvex at $\l \1$ then the second integral would by definition satisfy
\[  \int_{\om} P(\nabla u) \,dx \geq  \int_{\om} P(\l \1) \dx\] 
for any Lipschitz $u$ which agrees with $\ul$ on the boundary of $\om$.  This, when coupled with a straightforward approximation argument based on the estimate\footnote{This estimate follows from the fact that $A \mapsto \l_{i}(A)$ obeys $\l_{i}(A) \leq |A|$, which follows easily from the well-known fact that $\sum_{i=1}^{3} {\l_{i}}^{2}(A) = |A|^{2}$.}
\[|P(A)| \leq |A|^{2}+3 \l |A|,\] 
further implies 
\[  \int_{\om} P(\nabla u) \,dx \geq  \int_{\om} P(\l \1) \dx\] 
for any $u$ in $W^{1,q}(\om)$ with $q \geq 2$.
Finally, a short calculation shows that $P(\l \1)=0$, so that the right-hand side of the last inequality vanishes.  Thus the only condition needed in order to conclude that $I(u) \geq I(\ul)$ would be \eqref{j:suffcondpenultimate:3d},  which ensures the positivity of the integral involving $\scf^\l_{1}$.   

\end{document}